\documentclass[reqno]{amsart}
\usepackage{hyperref,xcolor,amssymb}
\usepackage[utf8]{inputenc}

\begin{document}
\title[Neumann $p$-Laplacian problems]
{Neumann $p$-Laplacian problems with a reaction term on metric spaces}

\author[A. Nastasi]
{Antonella Nastasi}

\address{Antonella Nastasi \newline
University of Palermo, Department of Mathematics and Computer Science, Via Archirafi 34, 90123, Palermo, Italy}
\email{antonella.nastasi@unipa.it}


\begin{abstract}
We use a variational approach to study existence and regularity of solutions for a Neumann $p$-Laplacian problem with a reaction term on metric spaces
equipped with a doubling measure and supporting a Poincaré inequality. 
Trace theorems for functions with bounded variation are applied in the definition of the variational functional and minimizers are shown to satisfy De Giorgi type conditions.
\end{abstract}

\maketitle
\numberwithin{equation}{section}
\newtheorem{theorem}{Theorem}[section]
\newtheorem{lemma}[theorem]{Lemma}
\newtheorem{proposition}[theorem]{Proposition}
\newtheorem{remark}[theorem]{Remark}
\newtheorem{definition}[theorem]{Definition}
\newtheorem{corollary}[theorem]{Corollary}
\allowdisplaybreaks

\section{Introduction}
The goal of this paper is to extend existence and regularity results for a Neumann boundary value problem valid on the Euclidean setting and, more generally, in Riemannian manifolds (see \cite{N}) to the general setting of metric spaces. 
The Neumann boundary value problem driven by a $p$-Laplacian operator is
\begin{equation}\label{P}
	\begin{cases}
		-\Delta_p u=g &\quad \mbox{in } \Omega, \\
		-|\nabla u|^{p-2}\partial_{\eta}u=f &\quad \mbox{on } \partial\Omega, 
	\end{cases}
\end{equation}
where $1<p<\infty$, $\Omega \subset \mathbb{R}^N$, $g$ is a continuous function and $\partial_{\eta}u$ is the directional derivative of $u$ in the direction of the outer normal to $\partial \Omega$.
The weak formulation of the problem is to find $u\in W^{1,p}(\Omega)$ such that
$$\int_{\Omega}|\nabla u(x)|^{p-2}\nabla u(x)\nabla \varphi(x)dx - \int_{\partial\Omega}\varphi(x) f(x)d\mathcal{H}^{n-1}(x)=\int_{\Omega}g(u(x))\varphi(x)dx,$$ for all $\varphi \in W^{1,p}(\Omega)$. 

Thus, solving $(\ref{P})$ reduces to look for critical points of the $p$-energy functional
$$J(u)= \int_{\Omega}|\nabla u|^{p}dx - \int_{\Omega}Gdx +\int_{\partial\Omega}uf d\mathcal{H}^{n-1},$$

where $G$ is a primitive of $g$. Minimizers of $J$ are solutions of the Neumann boundary value problem with the reaction term $G$ and boundary data $f$.


We apply variational methods such as those based on De Giorgi classes \cite{D} to consider $(\ref{P})$ in the setting of metric spaces.
The study of Sobolev spaces without a differentiable structure and of boundary value problems in metric measure spaces has attracted a lot of attention, see \cite{BB, BBS, C, FHK, H,HK, HK1, HKS, KM, KS, LMS, M, S}. In these papers, it is shown that the Sobolev spaces can be defined without using partial derivatives and a theory for $p$-Laplacian problems can be developed. 
 In particular, Kinnunen and Shanmugalingam \cite{KS} discuss regularity properties for the solutions of a Dirichlet problem in metric measure spaces. They use the De Giorgi method to prove that, if the space is equipped with a doubling measure and supports a Poincaré inequality, minimizers of the $p$-energy functional are locally H\"{o}lder continuous and they satisfy the Harnack inequality and the maximum principle. The study of Sobolev spaces on metric spaces can be applied in several areas of analysis, for example, calculus on Riemannian manifolds, subelliptic operators associated with vector fields, potential theory on graphs and weighted Sobolev spaces.

The energy functional for the Neumann problem involves an integral of the trace of a function of bounded variation, see \cite{HKLL,L, LS, M}. 
Recently, in the setting of metric measure spaces, Durand-Cartagena and  Lemenant \cite{DL},
 Lahti, Mal\'{y} and Shanmugalingam \cite{LMS} and Mal\'{y} and Shanmugalingam \cite{MS} have studied a Neumann problem obtaining existence and some regularity results of solutions. We study a Neumann boundary value problem as in \cite{LMS} and \cite{MS}, but the new feature is that we include a reaction term. Under appropriate conditions on the reaction term, we prove existence and boundedness properties of solutions  with a reaction term in a metric space equipped with a doubling measure and supporting a Poincaré inequality and thus extending the corresponding results in \cite{KS} and \cite{MS}.


 After an introduction and some useful notions which constitute the mathematical background (Sections 1 and 2), we obtain existence of a solution and a weaker uniqueness property (Section 3), that is the minimal $p$-weak upper gradients of the solutions of a Neumann boundary value problem with boundary data $f$ and reaction term $G$ are different at most on a set of measure zero. In Section 4, we prove that minimizers of this Neumann $p$-Laplacian problem satisfy a De Giorgi type inequality and consequently we give boundedness properties for them. Finally, in the last section we prove that minimizers of the Neumann $p$-Laplacian problem with zero boundary data are in the De Giorgi class. This permits us to conclude that the regularity results contained in \cite{KS} in the absence of a reaction term still hold true adding the reaction term $G$.

\section{Mathematical background}\label{Sec2}

Let $(X, d, \mu)$ be a metric measure space, where $\mu$ is a Borel regular measure. Let $B(x,\rho)\subset X$ be a ball with the center $x \in X$ and the radius $\rho>0$. For a measurable set $S \subset X$ of finite positive measure and for a measurable function $u: S \to \mathbb{R}$, we denote $$u_S= \dfrac{1}{\mu(S)}\int_{S} u d\mu.$$ We denote constants appearing in this paper with $K$, even if they assume different values.  
\begin{definition}[\cite{BB}, Section 3.1]
	A measure $\mu$ on $X$ is said to be doubling if there exists a constant $K$, called the doubling constant, such that 
	\begin{equation*}\label{doubling}
	0<\mu(B(x,2\rho))\leq K \mu(B(x,\rho))< +\infty,
	\end{equation*} for all $x \in X$ and $\rho>0$.
\end{definition}

\begin{definition}[\cite{BB}, Definition 1.13]
A non negative Borel measurable function $g$ is said to be an upper gradient of function $u: X \to [-\infty, +\infty]$ if, for all compact rectifiable arc lenght parametrized paths $\gamma$ connecting $x$ and $y$, we have
\begin{equation}\label{ug}
|u(x)-u(y)|\leq \int_{\gamma}g\, ds
\end{equation}
 whenever $u(x)$ and $u(y)$ are both finite and $\int_{\gamma}g \, ds= +\infty$ otherwise.
\end{definition}
As we have pointed out in the Introduction, the notion of upper gradient has been introduced in order to satisfy the lack of a differentiable structure.
We note that if $g$ is an upper gradient of function $u$ and $\phi$ is a non negative Borel measurable function, then $g+\phi$ is still an upper gradient of $u$. In order to overcome this aspect, we use the following notions that will lead to the definition of the minimal $p$-weak upper gradient of $u$.
\begin{definition}[\cite{BB}, Definition 1.33]
Let $p \in [1, +\infty[$. Let $\Gamma$ be a family of paths in $X$. We say that $$\inf_\phi \int_X \phi^p d\mu$$
is the $p$-modulus of $\Gamma$, where the infimum is taken among all non negative Borel measurable functions $\phi$ satisfying   $\int_\gamma \phi \, ds \geq 1,$ for all rectifiable paths $\gamma\in\Gamma$.
\end{definition}
Using the notion of $p$-modulus, we can deduce the following definition of $p$-weak uppper gradient.
\begin{definition}[\cite{BB}, Definition 1.32]
If \eqref{ug} is satisfied for $p$-almost all paths $\gamma$ in $X$, that is the set of non constant paths that do not satisfy \eqref{ug} is of zero $p$-modulus, then $g$ is said a $p$-weak uppper gradient of $u$.
\end{definition}
Clearly, every upper gradient is a weak upper gradient. The family of weak upper gradients satisfy the result contained in the following theorem concerning the existence of a minimal element.
\begin{theorem}[\cite{BB}, Theorem 2.5]\label{mpwug}
Let $p \in ]1, +\infty[$. Suppose that $u\in L^p(X)$ has an $L^p(X)$ integrable $p$-weak upper gradient. Then there exists a $p$-weak upper gradient, denoted with $g_u$, such that $g_u\leq g$ $\mu$-a.e. in $X$, for each $p$-weak upper gradient $g$ of $u$. This $g_u$ is called the minimal $p$-weak uppper gradient of $u$.
\end{theorem}
We note that $g_u$ is $\mu$-a.e. uniquely determinated by $u$. We also note that in the Euclidean setting, $g_u$ of Theorem \ref{mpwug}  assumes exactly the classical meaning of the modulus of the gradient of $u$. 

\begin{definition}[\cite{BB}, Definition 4.1]
	Let $p \in [1, +\infty[$. 
	A metric measure space $X$ supports a $(1, p)$-Poincar\'{e} inequality if there exist $K>0$ and $\lambda \geq 1 $ such that 
	\begin{equation*}
	\dfrac{1}{\mu(B(x,r))}\int_{B(x,r)} |u-u_{B(x,r)}|d\mu\leq K r \left(\dfrac{1}{\mu(B(x,\lambda r))}\int_{B(x,\lambda r)}g_u^p \, d\mu\right)^{\frac{1}{p}}
	\end{equation*}
	for all balls $B(x,r) \subset X$ and for all $u \in L^1_{loc}(X)$.
\end{definition}
 
Let $X$ be a complete metric space equipped with a doubling measure supporting a $(1, p)$-Poincaré inequality.
We recall the concept of Newtonian space, which is based on the notion of minimal $p$-weak uppper gradient. 
\begin{definition}
	The Newtonian space $N^{1,p}(X)$ is defined by $$N^{1,p}(X)=V^{1,p}(X)\cap L^p (X), \quad p\in [1,+\infty],$$ where
	$ V^{1,p}(X) = \{u: u \ \mbox {is measurable and }  g_u\in L^p (X) \}.$
	We consider $N^{1,p}(X)$ equipped with the norm 
	\begin{equation*}  \label{norm}
	\|u\|_{N^{1,p}(X)}=\|g_u\|_{L^p (X)} + \|u\|_{L^p (X)}.
	\end{equation*} 
We denote with $N^{1,p}_*(X)= \{u\in N^{1,p}(X):\int_{X}u \,dx=0\}$.
\end{definition}
The Newtonian space $N^{1,p}(X)$ is a complete normed vector space, which generalizes the Sobolev space $W^{1,p}(\Omega)$ to a metric setting.
	\begin{definition}[see \cite{M1}]
		A Borel set $E \subset X$ is said to be of finite perimeter if there exists a sequence $\{u_n\}_{n\in \mathbb{N}}$ in $N^{1,1}(X)$ such that $u_n \to \chi_E$ in $L^1(X)$ and $$\liminf_{n\to +\infty}\int_X g_{u_n} d\mu<\infty.$$ The perimeter $P_E(X)$ of $E$ is the infimum of the above limit among all sequences $\{u_n\}$ as above.
		For an open set $U\subset X$, the perimeter of $E$ in $U$ is $$P_E(U)=\inf\left\{\liminf_{n\to +\infty}\int_X g_{u_n}d\mu: \{u_n\}_{n\in \mathbb{N}} \subset N^{1,1}(U), u_n \to \chi_{E\cap U} \mbox{ in } L^1(U) \right\}.$$
	\end{definition}
We note that $E$ is a set of finite perimeter iff $\chi_{E}$ is a BV$(U)$ function (see \cite{M1}, Definition 4.1).

From now on, we consider a bounded domain (non empty, connected open set) $\Omega$ in $X$ with $X \setminus \Omega$ of positive measure such that $\Omega$ is of finite perimeter with perimeter measure $P_{\Omega}$. Let $f: \partial \Omega \to \mathbb{R}$ be a bounded $P_{\Omega}$-measurable function with $\int_{\partial \Omega}f dP_{\Omega}=0$.
\begin{lemma}[\cite{BB}, Lemma 3.3]\label{lemm3.3}
	Let $(\Omega, d, \mu)$ be a metric measure space with $\mu$ doubling. Then there is $s>0$ such that
	\begin{equation}\label{s}
	\dfrac{\mu(B(y,\rho))}{\mu(B(x, R))}\geq K\left(\dfrac{\rho}{R}\right)^s
	\end{equation}
	for all $\rho\in ]0, R]$, $x \in {\Omega}$, $y \in B(x, R)$ and some $K>0$.
\end{lemma}
We work under the same hypotheses set on ${\Omega}$ considered in \cite{MS}. 
Throughout the work, we make the following assumptions:
\begin{itemize}
	\item[$(H_1)$] There exists a constant $K\geq1$ such that for all $y \in {\Omega}$  and $0<\rho\leq {\rm diam}\,({\Omega})$, we have $$\mu(B(y, \rho) \cap {\Omega})\geq \dfrac{1}{K}\mu(B(y, \rho)).$$
	\item[$(H_2)$] (Ahlfors codimension 1 regularity of $P_{\Omega}$) For all $y \in \partial {\Omega}$ we have that $$\dfrac{1}{K\rho}\mu(B(y, \rho))\leq P_{\Omega}(B(y, \rho))\leq \dfrac{K}{\rho}\mu(B(y, \rho)),$$ where $K$ and $\rho$ are as in $(H_1)$.
	\item[$(H_3)$] $({\Omega}, d_{|{\Omega}}, \mu_{|{\Omega}})$ admits a $(1, p)$-Poincar\'{e} inequality with $\lambda=1$, where $p \in ]1, +\infty[$.
\end{itemize}

\begin{remark}
We point out some facts related to the above hypotheses, for reader convenience $($see \cite{MS}$)$. Indeed, we observe that $(H_1)$ and $(H_2)$ imply that $\mu(\partial {\Omega})=0$ and ${\Omega}$ is of finite perimeter. In addition, $(H_1)$-$(H_3)$ lead to the following Sobolev-type inequality for ${\Omega}$,
\begin{equation*}
\|u-u_{\Omega}\|_{L^p({\Omega})}\leq K \|g_{u}\|_{L^p({\Omega})},
\end{equation*} for some $K>0$.
In particular, if $u \in N^{1,p}_*({\Omega})$ then we get \begin{equation}\label{STI}
\|u\|_{L^p({\Omega})}\leq K \|g_{u}\|_{L^p({\Omega})}.
\end{equation}
\end{remark}

\begin{definition}[\cite{L}, Definition 4.1]
Let ${\Omega} \subset X$ be an open set and let $u$ be a $\mu$-measurable
function on ${\Omega}$. A function $Tu : \partial {\Omega} \to \mathbb{R}$ is the trace of $u$ if for $\mathcal{H}$-almost every $y \in \partial {\Omega}$ we have 
$$\lim_{\rho \to 0^+}\dfrac{1}{\mu({\Omega} \cap B(y,\rho))} \int_{{\Omega} \cap B(y,\rho)}|u - T u(y)| d\mu= 0.$$
\end{definition}
For the following existence theorem of the trace operator see \cite{MS} (further details can be found in \cite{LLW} and references therein).
\begin{theorem}[\cite{MS}, Proposition 3.8]
If $(H_1)$-$(H_3)$ hold true, there exists a bounded linear trace operator $$Tu: N^{1,p}({\Omega}) \to L^{\tilde p}(\partial {\Omega})$$ for every $\tilde p < \frac{p(s-1)}{s-p}$ if $p<s$ and $\tilde p < + \infty$ if $p\geq s$. This trace operator is given as follows. For $u \in N^{1,p}({\Omega})$, $\mathcal{H}$-almost every $y \in \partial {\Omega}$, there exists $Tu(y) \in \mathbb{R}$ such that $$\lim_{\rho \to 0^+} \frac{1}{\mu(B(y, \rho)\cap {\Omega})}\int_{B(y, \rho)\cap {\Omega}}|u-Tu(y)|d\mu=0.$$
Here $s$ is the lower exponent of \eqref{s}.
\end{theorem}

We remark that, from the results in \cite{AMP} and \cite{L1}, if ${\Omega}$ supports a $(1, p)$-Poincar\'{e} inequality, then  $P_{\Omega} \approx \mathcal{H}_{|\partial {\Omega}}$, where $\mathcal{H}$ stands for the codimension one Hausdorff measure. For further details see \cite{A, KKST} and references therein.

Given a Neumann boundary value problem with boundary data $f\neq 0$ and reaction term $G$, we associate the following functional
\begin{equation}\label{J}
J(u)= \int_{\Omega} g_{u}^p d \mu -\int_{\Omega}G(u) d\mu+ \int_{\partial {\Omega}} Tu f dP_{\Omega} \quad \mbox{for all $u \in N^{1,p}(\Omega)$}.
\end{equation}

\begin{definition}
A function $u_0\in N^{1,p}_*({\Omega})$ is a $p$-harmonic solution to the Neumann boundary value problem with boundary data $f \neq 0$ and reaction term $G$ if 
\begin{align*}\label{min}
J(u_0)=& \int_{\Omega} g_{u_0}^p d \mu -\int_{\Omega}G(u_0) d\mu+ \int_{\partial {\Omega}} Tu_0 f dP_{\Omega} \nonumber \\
\leq&\int_{\Omega} g_{v}^p d \mu -\int_{\Omega}G(v) d\mu+ \int_{\partial {\Omega}} Tv f dP_{\Omega} = J(v)
\end{align*} 
for every $v \in N^{1,p}_*({\Omega})$, where $g_{u_0}$, $g_{v}$ are the minimal $p$-weak upper gradients of $u_0$ and $v$ in ${\Omega}$, respectively, and $Tu_0$ and $Tv$ are the traces of $u_0$ and $v$ on $\partial {\Omega}$, respectively.
\end{definition}

Throughout the paper, in considering the trace $Tu$ of $u$ we will omit $T$ and just write $u$.
The following theorem gives an embedding result.
\begin{theorem}[\cite{BB}, Theorem 5.50]
Assume that ${\Omega}$ supports a $(1, p)$-Poincaré
inequality, that \eqref{s} holds and that there exist $y_0 \in {\Omega}$ and a sequence $\{\rho_n\}_{n\in \mathbb{N}}$ such that $\lim_{n \to +\infty}\rho_n= +\infty$ and $\mu(B(y_0,\rho_n))\geq K\rho_n^s $ for all $n\in \mathbb{N}$.
\begin{itemize}
\item[(i)] If $s > p$, then $N^{1,p}({\Omega})$ continuously embeds into $L^{p*}({\Omega})$ with $p*= \frac{sp}{s-p}$.
\item[(ii)] If $s < p$ and ${\Omega}$ is complete, then $N^{1,p}({\Omega})$ continuously embeds into $C^{1-\frac{s}{p}}({\Omega})$.	
\end{itemize}
\end{theorem}

Here, we assume that $G:{\Omega} \to \mathbb{R}$ is defined as follows.
\begin{equation}\label{G}
G(u)= c -|u|^{\gamma}  \mbox{ for all } u \in  N^{1,p}({\Omega}), 
\end{equation}
for some $c>0$ and  $1<\gamma<p^*=\frac{ps}{s-p}$ if $p<s$ and $1<\gamma< +\infty$ otherwise.

In the metric setting, we will look for a minimizer of $J$ in the Newtonian space $N^{1,p}_*({\Omega})$.
\begin{remark}\label{R0}
We note that, from the definition of $J$, $$\inf_{u \in N^{1,p}_*({\Omega})}J(u)\leq J(0)= -c \mu({\Omega})<0.$$
\end{remark}
\begin{lemma}\label{lemma3.3}
If $u\in N^{1,p}_*({\Omega})$ and $f\in L^q(\partial {\Omega})$, with $q=1$ if $p>s$ and 
$\frac{p(s-1)}{s(p-1)}<q\le +\infty$ if $p\leq s$. Then there is a constant $K>0$ such that
$$J(u)\geq \|g_{u}\|_{L^p({\Omega})}\left(\|g_{u}\|^{p-1}_{L^p({\Omega})}-K \|f\|_{L^{q}(\partial {\Omega})}\right)-c \mu({\Omega}).$$
\end{lemma}

\begin{proof} By H\"{o}lder inequality
\begin{align*}
	J(u)&= \int_{\Omega}g_{u}^{p}d\mu-\int_{\Omega}(c -|u|^{\gamma}) d\mu+\int_{\partial {\Omega}}uf dP_{\Omega}\nonumber \\& \geq \int_{\Omega}g_{u}^{p}d\mu-c \mu({\Omega})-\int_{\partial {\Omega}}|uf| dP_{\Omega}\nonumber \\&\geq \|g_{u}\|^{p}_{L^p({\Omega})}-c \mu({\Omega})-\|u\|_{L^{q'}(\partial {\Omega})} \|f\|_{L^{q}(\partial {\Omega})}.
\end{align*}
It is known that $\|u\|_{L^{q'}(\partial {\Omega})}\leq K\|g_{u}\|_{L^p({\Omega})}$  (since $u \in N^{1,p}_*({\Omega})$, see \cite{M}), so we conclude that
\begin{equation}\label{2.6}
J(u)\geq \|g_{u}\|_{L^p({\Omega})}\left(\|g_{u}\|^{p-1}_{L^p({\Omega})}-K \|f\|_{L^{q}(\partial {\Omega})}\right)-c \mu({\Omega}).
\end{equation}
\end{proof}

\begin{corollary}\label{P32}
For every $u \in N^{1,p}_*({\Omega})$, there is a constant $K>0$ such that $$J(u)\geq -K\|f\|^{p'}_{L^q(\partial {\Omega})}-c \mu({\Omega}).$$ 
\end{corollary}
\begin{proof}
It is sufficient to consider the minimum of the following function $$[0, +\infty[\ni t \mapsto t^p-Kt\|f\|_{L^{q}(\partial {\Omega})}.$$
\end{proof}

\section{Existence of a solution}\label{Sec3}
The existence of a nontrivial solution to the Neumann boundary value problem with non zero boundary data $f$ and reaction term $G$ is an immediate consequence of the following theorem which shows that $J$ has a minimizer.
\begin{theorem}
Let $J$ defined as in \eqref{J}. Then $J$ has a minimizer in $N^{1,p}_*({\Omega})$.
\end{theorem}
\begin{proof}
Let $J_0=\inf_{u\in N^{1,p}_*({\Omega})} J(u)$. If $J_0=J(0)$, then $0$ is a minimizer. So, we assume that $J_0<J(0)$. We consider a sequence $\{u_n\}\subset N^{1,p}_*({\Omega})$ such that $$J_0=\lim\limits_{n\to +\infty}J(u_n) \quad \mbox{and} \quad J(u_n)\leq -c \mu({\Omega}),$$ for all $n \in \mathbb{N}$ (see Remark \ref{R0}).
For all $n \in \mathbb{N}$, we consider the $p$-weak minimal upper gradients of $u_n$, say $g_{u_n}$. By \eqref{2.6}, we deduce that $$-c \mu({\Omega})\geq J(u_n)\geq \|g_{u_n}\|_{L^p({\Omega})}(\|g_{u_n}\|^{p-1}_{L^p({\Omega})}-K \|f\|_{L^{q}(\partial {\Omega})})-c \mu({\Omega}),$$ that is $$\|g_{u_n}\|^{p-1}_{L^p({\Omega})}\leq K \|f\|_{L^{q}(\partial {\Omega})}.$$ This ensures that the sequence of minimal $p$-weak upper gradients $\{g_{u_n}\}$ is bounded in $L^p({\Omega})$. 
We recall that the Sobolev-type inequality \eqref{STI} holds.
As a consequence we can also deduce that $\{u_n\}$ is bounded. Since $L^p({\Omega})$ is a reflexive space, there are two subsequences, still denoted with $\{g_{u_n}\}$ and $\{u_n\}$, weakly convergent to some elements $g_{u}$ and $u$, respectively.
Using Mazur's Lemma, we can find two convex combinations $\widehat{u}_n= \sum_{i=n}^{N(n)} \alpha_{n,i}u_i$ and $g_{\widehat{u}_n}= \sum_{i=n}^{N(n)} \alpha_{n,i}g_{u_i}$ with $n=1,2,...$ such that $\{\widehat{u}_n\}$ converges to $\widehat{u}$ and $\{g_{\widehat{u}_n}\}$ converges to $g_{\widehat{u}}$ as $n \to +\infty$. Proceeding as in the proof of Theorem 4.3 of \cite{MS}, we have that $\widehat{u} \in N^{1,p}_*({\Omega})$ and $$\int_{\Omega} g_{\widehat{u}}^p \, d\mu \leq \liminf _{n\to+\infty}\int_{\Omega} g_{\widehat{u}_n}^p d\mu,$$ where $g_{\widehat{u}}$ and $g_{\widehat{u}_n}$ are the minimal $p$-weak upper gradients of $\widehat{u}$ and $\widehat{u}_n$, respectively.
Since the trace operator is linear and since functions $$\widehat{u} \mapsto \int_{\Omega} |\widehat{u}|^{\gamma} d\mu \quad \mbox{and} \quad \widehat{u} \mapsto \int_{\Omega} g_{\widehat{u}}^p d\mu$$ are convex, we conclude that $J$ is convex. Therefore,
\begin{align*}
J_0 &\leq J(\widehat{u}_n)=J\Big(\sum_{i=n}^{N(n)} \alpha_{n,i}u_i\Big)\leq \sum_{i=n}^{N(n)} \alpha_{n,i}J(u_i)\to J_0  \quad \mbox{as } n\to +\infty.\nonumber \\&
\end{align*}
Thus,
\begin{align*}
J_0 &\leq J(\widehat{u})=\int_{\Omega} g_{\widehat{u}}^p \, d\mu-\int_{\Omega} (c-|\widehat{u}|^{\gamma}) d\mu+\int_{\partial {\Omega}}\widehat{u} f dP_{\Omega}\nonumber \\&\leq \liminf _{n\to +\infty}\left(\int_{\Omega} g_{\widehat{u}_n}^p \, d\mu-\int_{\Omega} (c-|\widehat{u}_n|^{\gamma}) d\mu+\int_{\partial {\Omega}}\widehat{u}_n	 f dP_{\Omega}\right)\nonumber \\&=\liminf _{n\to+\infty}J(\widehat{u}_n)=J_0.
\end{align*}
So $J$ has $\widehat{u}$ as minimizer in $N^{1,p}_*({\Omega})$. 
\end{proof}

\begin{proposition}
	Let  $\mathcal{M}= \{ u \in N^{1,p}_*({\Omega}): J(u)=J_0\}$ be the set of minimizers of $J$. Then $\mathcal{M}$ is norm-closed and convex.
\end{proposition}
\begin{proof}
	We consider $t \in (0, 1)$ and $u_1, u_2 \in \mathcal{M}$.
	Since, as we have already pointed out, $J$ is a convex functional, we have that 
	\begin{align*}
	J(tu_1+(1-t)u_2)&\leq tJ(u_1) + (1-t)J(u_2)\\&=tJ_0 + (1-t)J_0=J_0.
	\end{align*} 
That means that $tu_1+(1-t)u_2\in \mathcal{M}$ and so, $\mathcal{M}$ is convex.
	From the sequential lower semi-continuity of $J$ we have also that $\mathcal{M}$ is norm-closed.
\end{proof}

\begin{proposition}
	Let $u_1,u_2 \in \mathcal{M}$. Then
	\begin{itemize}
		\item[(i)] $-\int_{{\Omega}}G(u_1) d\mu+ \int_{\partial {\Omega}} u_1 f dP_{\Omega} = -\int_{\Omega}G(u_2) d\mu+ \int_{\partial {\Omega}} u_2 f dP_{\Omega}$;
		\item[(ii)] $g_{u_1}=g_{u_2}$ a.e. in ${\Omega}$.
	\end{itemize}  
\end{proposition}
\begin{proof}
Firstly we note that, from $J(u_1)=J(u_2)$ and $(ii)$, we deduce that $(i)$ holds. So, in order to conclude the proof, we just need to prove that $(ii)$ is satisfied. By absurd, let $\mu(\{x \in {\Omega} :g_{u_1}\neq g_{u_2}\}) >0$. Then we can choose $\delta>0$ such that $D_{\delta}= \{x \in {\Omega} :|g_{u_1}-g_{u_2}|>\delta\}$ has positive measure. 
We consider $$u=\dfrac{u_1+u_2}{2}.$$ From the definition of minimal $p$-weak upper gradient, we have that $$g_{u}\leq \dfrac{g_{u_1}+g_{u_2}}{2}.$$ The function $t\mapsto t^p$	is uniformly convex on $[0, +\infty[$. Thus there exists $\epsilon= \delta^pp (2^{-1}-2^{-p})$ such that $$\left(\frac{g_{u_1}+g_{u_2}}{2}\right)^p\leq \frac{g_{u_1}^p+g_{u_2}^p}{2}-\epsilon,$$
where $|g_{u_1}-g_{u_2}|\geq\delta$.
As a consequence, we get that 
\begin{align*}
	J(u)=& \int_{\Omega} g_{u}^p d\mu-\int_{\Omega} (c-|u|^{\gamma}) d\mu+\int_{\partial {\Omega}}u f dP_{\Omega}\nonumber \\
 \leq&\int_{D_{\delta}}\left(\frac{g_{u_1}^p+g_{u_2}^p}{2}- \epsilon\right)d\mu+\int_{{\Omega}\setminus D_{\delta}}\frac{g_{u_1}^p+g_{u_2}^p}{2}d\mu\nonumber \\ &+\int_{\Omega}\frac{|u_1|^{\gamma}+|u_2|^{\gamma}}{2}d\mu-c\mu({\Omega})+\int_{\partial {\Omega}}\dfrac{u_1+u_2}{2} f dP_{\Omega}\nonumber \\
 =&\frac{1}{2}\left(\int_{\Omega}g_{u_1}^p d\mu+\int{\Omega} |u_1|^{\gamma} d\mu-c\mu({\Omega})+\int_{\partial {\Omega}}u_1 f dP_{\Omega}\right) 
 \\&+\frac{1}{2}\left(\int_{\Omega}g_{u_2}^p d\mu+\int_{\Omega} |u_2|^{\gamma} d\mu-c\mu({\Omega})+\int_{\partial {\Omega}}u_2 f dP_{\Omega}\right) 
 -\epsilon \mu (D_{\delta})\nonumber \\ =& J(u_1)-\epsilon \mu (D_{\delta}).
\end{align*}
Since $u_1 \in \mathcal{M}$, $J(u_1) \leq J(u) \leq J(u_1)-\epsilon \mu (D_{\delta})$ that is absurd. Thus, we have proven $(ii)$. 
\end{proof} 
\section{Boundedness property }
In this section we show that minimizers are locally bounded near the boundary under appropriate hypothesis on the boundary data $f$. In the absence of the reaction term $G$ this result has been proven in \cite{MS}.

Let ${\Omega}$ be a bounded domain such that hypotheses $(H_1)$-$(H_3)$ hold.
Let $u_0\in N^{1,p}_*({\Omega})$ be a minimizer of 
\begin{equation*}
J(u)= \int_{\Omega} g_{u}^p d \mu -\int_{\Omega}(c -|u|^{\gamma}) d\mu+ \int_{\partial {\Omega}} u f dP_{\Omega}, \quad \mbox{$u \in N^{1,p}_*({\Omega})$.}
\end{equation*} 
In this section we assume that $f \in L^{\infty}(\partial {\Omega})$. Our aim is to prove that, under this assumption, we get that $u\in L^{\infty}({\Omega}_R)$ and $Tu \in L^{\infty}(\partial {\Omega}_R)$ where 
\begin{equation}\label{omegaR}
\Omega_R=\left\{y\in {\Omega} : d(y, \partial {\Omega})<\frac{R}{2}\right\}
\end{equation} 
for an appropriate $R>0$, that is $u$ is bounded near the boundary. In the following lemma we give a De Giorgi type inequality which permits to use the De Giorgi method to conclude on the local boundedness of minimizers.
\begin{lemma}\label{lem 4.1}
Let $u\in N^{1,p}_*({\Omega})$ be a minimizer of $J$ and $f \in L^{\infty}(\partial {\Omega})$. If $y \in \partial {\Omega}$, $0<\rho<R<\frac{{\rm diam}({\Omega})}{10}$ and $\alpha \in \mathbb{R}$, then there is $K\geq 1$ such that the following De Giorgi type inequality 
\begin{align}\label{5.4}
\int_{{\Omega}\cap B(y,\rho)}g_{(u-\alpha)_+}^p d \mu \leq& \frac{K}{(R-\rho)^p} \int_{{\Omega}\cap B(y,R)} (u-\alpha)_+^p d \mu \\ \nonumber
&+ K \int_{\partial {\Omega}\cap B(y,R)} |f|(u-\alpha)_+^p dP_{\Omega}
\end{align}
is satisfied.
\end{lemma}

\begin{proof}
We define 
	\begin{equation*} \label{funzione di troncamento}
	\tau_{\rho, R}(x)= \tau (x)= \left(1-\frac{d(x, B(y,\rho))}{R-\rho}\right)_+
	\end{equation*}
and
\begin{equation*}\label{S}
S_{\alpha, r}= \{x \in B(y, r)\cap {\Omega}: u(x)>\alpha\} \cup \{x \in B(y, r)\cap \partial {\Omega}: u(x)>\alpha\}.
\end{equation*}	
We consider 
\begin{equation}\label{v}
w=u- \tau (u-\alpha)_+= \begin{cases}
(1-\tau)(u-\alpha)+\alpha \quad \mbox{ in } S_{\alpha, R}\\
u \quad \quad \quad\quad\quad\quad\quad\quad\mbox{ otherwise.}\\
\end{cases}
\end{equation}
We observe that, from the definition of $w$, we have  $|w|\leq |u|$.
Using Leibniz rule,
\begin{equation}\label{nablav}
g_{w}\leq \begin{cases}
(1-\tau)g_{u}+\dfrac{u-\alpha}{R-\rho} \chi_{B(y, R)\setminus B(y,\rho)} \quad \mbox{ in } S_{\alpha, R}\\
g_{u} \quad \quad \quad\quad\quad\quad\quad\quad\mbox{ otherwise.}
\end{cases}
\end{equation}	
By (\ref{nablav}) we deduce that
\begin{equation}\label{nablavbis}
g_{w} ^p\leq 
2^p\left(g_{u} ^p(1-\chi_{S_{\alpha, r}})+\dfrac{(u-\alpha)^p}{(R-\rho)^p}\right) \quad \mbox{ in } S_{\alpha, R}.
\end{equation}	
Since $u$ is a minimizer of $J$, then 
\begin{align}\label{5.7}
J(u)&= \int_{{\Omega}\cap B(y,R)}g_{u} ^p d \mu -\int_{{\Omega}\cap B(y,R)}(c -|u|^{\gamma}) d\mu+ \int_{\partial {\Omega}\cap B(y,R)} u f dP_{\Omega} \nonumber \\
&\leq\int_{{\Omega} \cap B(y,R)} g_{w} ^p d \mu -\int_{{\Omega}\cap B(y,R)}(c -|w|^{\gamma}) d\mu+ \int_{\partial {\Omega}\cap B(y,R)} w f dP_{\Omega} = J(w).
\end{align}	
By adding $$-\int_{{\Omega}\cap B(y,R)\setminus S_{\alpha, R}} g_{u} ^p d \mu +\int_{{\Omega}\cap B(y,R)}(c -|u|^{\gamma}) d\mu- \int_{\partial {\Omega}\cap B(y,R)} u f dP_{\Omega}$$ to both sides of (\ref{5.7}), we get 
\begin{align}\label{5.7bis}
\int_{S_{\alpha, R}} g_{u} ^p d \mu  \leq&\int_{S_{\alpha, R}} g_{w} ^p d \mu -\int_{{\Omega}\cap B(y,R)}(c -|w|^{\gamma}-(c -|u|^{\gamma})) d\mu\nonumber \\&- \int_{\partial {\Omega}\cap S_{\alpha, R}} \tau(u-\alpha) f dP_{\Omega}\nonumber \\
 \leq&\int_{S_{\alpha, R}} g_{w} ^p d \mu -\int_{{\Omega}\cap B(y,R)}(|u|^{\gamma}-|w|^{\gamma}) d\mu- \int_{\partial {\Omega}\cap S_{\alpha, R}} \tau(u-\alpha) f dP_{\Omega}\nonumber \\
 \leq&\int_{S_{\alpha, R}} g_{w} ^p d \mu - \int_{\partial {\Omega}\cap S_{\alpha, R}} \tau(u-\alpha) f dP_{\Omega} \quad\mbox{(by $(\ref{v})$).}
\end{align}	
Using $(\ref{nablavbis})$ and $(\ref{5.7bis})$, we obtain
\begin{align*}
\int_{S_{\alpha, \rho}} g_{u} ^p d \mu	\leq 2^p\int_{S_{\alpha, R}\setminus S_{\alpha, \rho}} g_{u} ^p d \mu &+ \dfrac{2^p}{(R-\rho)^p}\int_{S_{\alpha, R}} (u-\alpha)^p	d \mu\nonumber \\
& -\int_{\partial {\Omega}\cap S_{\alpha, R}} \tau(u-\alpha) f dP_{\Omega}.
\end{align*}	
Now, we add $2^p\int_{S_{\alpha, \rho}} g_{u} ^p d \mu$ to both sides of the inequality, then we divide all by $1+2^p$ and we obtain

\begin{align}\label{5.8}
	\int_{S_{\alpha, \rho}} g_{u} ^p d \mu	\leq \dfrac{2^p}{1+2^p}\int_{S_{\alpha, R}} g_{u} ^p d \mu &+ \dfrac{2^p}{(1+2^p)(R-\rho)^p}\int_{S_{\alpha, R}} (u-\alpha)^p	d \mu\nonumber \\
	& -\dfrac{1}{(1+2^p)}\int_{\partial {\Omega}\cap S_{\alpha, R}} \tau(u-\alpha) f dP_{\Omega}.
\end{align}	
At this point we can use $(\ref{5.8})$ and Lemma 6.1 of \cite{G} to get
\begin{align*}
	\int_{S_{\alpha, \rho}} g_{u} ^p d \mu	\leq \dfrac{K}{(R-\rho)^p}\int_{S_{\alpha, R}} (u-\alpha)^p	d \mu+K\int_{\partial {\Omega}\cap S_{\alpha, R}} \tau(u-\alpha) |f| dP_{\Omega},
\end{align*}
that is $(\ref{5.4})$ holds true. 	
\end{proof}

\begin{theorem}\label{4.1}
	Let  $0<R<\frac{{\rm diam}({\Omega})}{4}$ 
	and ${\Omega}_R$ as in \eqref{omegaR}.
	If $u\in N^{1,p}_*({\Omega})$ is a minimizer of $J$ and $f \in L^{\infty}(\partial {\Omega})$, then $u \in L^{\infty}({\Omega}_R)$ and $Tu\in L^{\infty}(\partial {\Omega}_R)$.
\end{theorem}
\begin{proof}
Proceeding as in the proof of Theorem 5.2 of \cite{MS}, we can find $d \geq 0$
such that $$\int_{{\Omega} \cap B\left(x, \frac{R}{2}\right)} (u-d)_+^p d \mu=0, \quad \mbox{for all $x \in \partial {\Omega}$}.$$ This implies that $u \leq d$ $\mu-$a.e. in ${\Omega} \cap B\left(x, \frac{R}{2}\right)$. Consequently, $u \leq d$ $\mu-$a.e. in ${\Omega}_R$.
In order to deduce that $u$ is also $\mu-$a.e. lower bounded, we observe that if $u$ is a minimizer for $J$, then $-u$ is a minimizer for $J_{-}$, where $J_{-}$ is defined as $$J_{-}(u)= \int_{\Omega} g_{u} \, d \mu - \int_{\Omega}  (c-|u|^{\gamma}) d \mu - \int_{\partial {\Omega}} u f d P_{\Omega}.$$ In fact, $u$ minimizer for $J$  means $J(u)\leq J(v)$ for all $v \in N^{1,p}_*({\Omega})$. We have that $J_{-}(-u)=J(u)\leq J(v)=J_{-}(-v)$ for all $v \in N^{1,p}_*({\Omega})$, which means that $-u$ is a minimizer of $J_{-}$. This ensures that $-u$ is $\mu-$a.e. upper bounded in ${\Omega}_R$ and so $u$ is $\mu-$a.e. lower bounded in ${\Omega}_R$. We conclude that $u \in L^{\infty}({\Omega}_R)$. In a similar way, we have that $Tu\in L^{\infty}(\partial {\Omega}_R)$. 
\end{proof}

\section{Neumann p-Laplacian problem with zero boundary data}
In this section we consider a Neumann p-Laplacian problem in the particular case of zero boundary data, that is $f=0$.  
In this case, the functional corresponding to the problem is given by $$I(u)= \int_{\Omega} g_{u} ^p \, d \mu -\int_{\Omega}G d\mu\quad \mbox{for all $u \in N^{1,p}(\Omega)$},$$
where $G$ is defined by \eqref{G}. Clearly, $I$ is bounded from below and sequentially lower semi continuous and so we can deduce the existence of a minimizer. Now we see how some regularity results concerning solutions of such a problem can be obtained as in \cite{KS}.

We introduce the following definition.
\begin{definition}
	Let ${\Omega}$ be an open subset of $X$. Let $u \in N^{1,p}({\Omega})$ for which there exists $K>0$ such that 
\begin{align}\label{3.31}
	\int_{B(y,\rho)}g_{ (u-\alpha)_+}^p d \mu \leq & \frac{K}{(R-\rho)^p} \int_{B(y,R)} (u-\alpha)_+^p d \mu, 
\end{align} for all $\alpha \in \mathbb{R}$, $y \in {\Omega}$ and $0<\rho<R<\frac{{\rm diam}({\Omega})}{10}$ so that $B(y, R) \subset {\Omega}$. 
Then we say that $u$ is in the De Giorgi class and write $u \in DG_p({\Omega})$.
\end{definition}
Proceeding as in the proof of Lemma \ref{lem 4.1},  we obtain the following result.
\begin{lemma}\label{lemDGp}
	If $u \in N^{1,p}({\Omega})$ is a minimizer of $I$, then $u \in DG_p({\Omega})$.
\end{lemma}
\begin{proof}
We define 
	\begin{equation*} \label{funzione di troncamento1}
	\tau_{\rho, R}(x)= \tau (x)= \left(1-\frac{d(x, B(y,\rho))}{R-\rho}\right)_+
	\end{equation*}
and
\begin{equation*}\label{S1}
S_{\alpha, r}= \{x \in B(y, r): u(x)>\alpha\}.
\end{equation*}	
We consider 
\begin{equation}\label{v1}
v=u- \tau (u-\alpha)_+= \begin{cases}
(1-\tau)(u-\alpha)+\alpha \quad \mbox{ in } S_{\alpha, R}\\
u \quad \quad \quad\quad\quad\quad\quad\quad\mbox{ otherwise.}\\
\end{cases}
\end{equation}
We observe that, from the definition of $v$, we have  $|v|\leq |u|$.
Using Leibniz rule,
\begin{equation}\label{nablav1}
g_{v} \leq \begin{cases}
(1-\tau)g_{u} +\dfrac{u-\alpha}{R-\rho} \chi_{B(y, R)\setminus B(y,\rho)} \quad \mbox{ in } S_{\alpha, R}\\
g_{u}  \quad \quad \quad\quad\quad\quad\quad\quad\mbox{ otherwise.}
\end{cases}
\end{equation}	
By (\ref{nablav1}) we deduce that
\begin{equation}\label{nablavbis1}
g_{v} ^p\leq 
2^p\left(g_{u} ^p(1-\chi_{S_{\alpha, r}})+\dfrac{(u-\alpha)^p}{(R-\rho)^p}\right) \quad \mbox{ in } S_{\alpha, R}.
\end{equation}	
Since $u$ is a minimizer of $J$, then 
\begin{align}\label{5.71}
J(u)&= \int_{B(y,R)} g_{u} ^p \, d \mu -\int_{B(y,R)}(c -|u|^{\gamma}) d\mu\nonumber \\
&\leq\int_{B(y,R)} g_{v} ^p\,  d \mu -\int_{B(y,R)}(c -|v|^{\gamma}) d\mu= J(v).
\end{align}	
By adding $-\int_{B(y,R)\setminus S_{\alpha, R}} g_{u} ^p \, d \mu +\int_{ B(y,R)}(c -|u|^{\gamma}) d\mu$ to both sides of (\ref{5.71}), we get 
\begin{align}\label{5.7bis1}
\int_{S_{\alpha, R}} g_{u} ^p \, d \mu  \leq&\int_{S_{\alpha, R}} g_{v} ^p\, d \mu -\int_{B(y,R)}(c -|v|^{\gamma}-(c -|u|^{\gamma})) d\mu\nonumber \\
 \leq &\int_{S_{\alpha, R}} g_{v} ^p\, d \mu -\int_{B(y,R)}(|u|^{\gamma}-|v|^{\gamma}) d\mu\nonumber \\
 \leq&\int_{S_{\alpha, R}} g_{v} ^p\, d \mu  \quad\mbox{(by $(\ref{v1})$).}
\end{align}	
Using $(\ref{nablavbis1})$ and $(\ref{5.7bis1})$, we obtain
\begin{align*}
\int_{S_{\alpha, \rho}} g_{u} ^p\, d \mu	\leq 2^p\int_{S_{\alpha, R}\setminus S_{\alpha, \rho}} g_{u} ^p \, d \mu &+ \dfrac{2^p}{(R-\rho)^p}\int_{S_{\alpha, R}} (u-\alpha)^p	d \mu.
\end{align*}	
Now, we add $2^p\int_{S_{\alpha, \rho}}g_{u}^p \, d \mu$ to both sides of the inequality, then we divide all by $1+2^p$ and we obtain

\begin{align}\label{5.81}
	\int_{S_{\alpha, \rho}} g_{u}^p \, d \mu	\leq \dfrac{2^p}{1+2^p}\int_{S_{\alpha, R}} g_{u}^p \, d \mu &+ \dfrac{2^p}{(1+2^p)(R-\rho)^p}\int_{S_{\alpha, R}} (u-\alpha)^p	d \mu.
\end{align}	
At this point we can use $(\ref{5.81})$ and Lemma 6.1 of \cite{G} to get
\begin{align*}
	\int_{S_{\alpha, \rho}} g_{u}^p \, d \mu	\leq \dfrac{K}{(R-\rho)^p}\int_{S_{\alpha, R}} (u-\alpha)^p	d \mu,
\end{align*}
that is equivalent to \eqref{3.31} and so $u \in DG_p({\Omega})$. 	
\end{proof}
We observe that if $u$ is a minimizer of $I$ then $-u$ is also a minimizer of $I$ and so, by Lemma \ref{lemDGp}, $-u \in DG_p({\Omega})$. Consequently, thanks to the results proven in \cite{KS}, we can conclude that minimizers of $I$, and so solutions of  a Neumann problem with reaction term $G$ and zero boundary data, are locally H\"{o}lder continuous and they satisfy Harnack's inequality and the strong maximum principle (see Sections 5, 6 and 7 of \cite{KS}).

\end{document}